\documentclass[11pt,centertags,oneside]{article}
\usepackage{amsmath,amstext,amsthm,amscd,typearea}
\usepackage{amssymb}
\usepackage{a4wide}
\usepackage[mathscr]{eucal}
\usepackage{mathrsfs}
\usepackage{typearea}
\usepackage{charter}
\usepackage{pdfsync}
\usepackage{url}

\usepackage[a4paper,width=16.2cm,top=3cm,bottom=3cm]{geometry}

\numberwithin{equation}{section}

\newtheorem{theorem}{Theorem}[section]

\newtheorem{proposition}[theorem]{Proposition}
\newtheorem{corollary}[theorem]{Corollary}
\newtheorem{lemma}[theorem]{Lemma}

\newcommand{\cali}[1]{\mathscr{#1}}

\newcommand{\supp}{{\rm Supp}}

\newcommand{\dist}{\mathop{\mathrm{dist}}\nolimits}

\newcommand{\ddc}{\text{\normalfont dd}^c}
\newcommand{\dc}{\text{\normalfont d}^c}
\def\d{\operatorname{d}}

\newcommand{\Ta}{\text{\normalfont T}}

\newcommand{\id}{{\rm id}}

\newcommand{\ord}{{\rm ord}}

\newcommand{\codim}{{\rm codim\ \!}}

\newcommand{\C}{\mathbb{C}}
\newcommand{\D}{\mathbb{D}}
\newcommand{\N}{\mathbb{N}}

\newcommand{\R}{\mathbb{R}}
\renewcommand\P{\mathbb{P}}

\newcommand{\explain}[1]{\text{\scriptsize\sf [#1]}}
\title{\bf On the set of divisors with zero geometric defect}
\providecommand{\keywords}[1]{\textbf{\textit{Keywords:}} #1}
\providecommand{\subject}[1]{\textbf{\textit{Mathematics Subject Classification 2010:}} #1}

\author{Dinh Tuan Huynh and Duc-Viet Vu}

\newcommand{\Addresses}{{
		\bigskip
		\footnotesize
		\textsc{Dinh Tuan Huynh, Max Planck Institute for Mathematics, Vivatsgasse 7, 53111 Bonn, Germany \& Department of Mathematics, College of Education, Hue University, 34 Le Loi St., Hue City, Vietnam}
		
		\par\nopagebreak
		\noindent
		\textit{E-mail address}: \texttt{dinhtuanhuynh@mpim-bonn.mpg.de}
		\newline
		
		\noindent
		\textsc{Duc-Viet Vu, University of Cologne, Mathematical Institute, Weyertal 86-90, 50931, K\"oln,  Germany  \& Thang Long Institute of Mathematics and Applied Sciences, Hanoi, Vietnam}
		\noindent
		\par\nopagebreak
		\noindent
		\textit{E-mail address}: \texttt{vuduc@math.uni-koeln.de}	
}}

\date{\today}
\begin{document}
\maketitle
\begin{abstract}   Let $f: \C \to X$ be  a transcendental holomorphic curve into a complex projective manifold $X$. Let $L$ be a very ample line bundle on $X.$ Let $s$ be a very generic holomorphic section of $L$ and $D$ the zero divisor given by $s.$ We prove that   the \emph{geometric} defect of $D$ (defect of truncation $1$) with respect to $f$ is zero.  We also prove that $f$ almost misses  general enough analytic subsets on $X$ of codimension $2$.   
\end{abstract}
\noindent
, {tangent current}, {density current}, {holomorphic curve}

\noindent
\subject{32U40}, {32H50}, {37F05}


\section{Introduction} \label{introduction}

Let $X$ be a compact K\"ahler manifold  and $D$ an effective divisor in $X.$  Let $f: \C \to X$ be  a  holomorphic curve such that $f(\C) \not \subset \supp D$. In Nevanlinna's theory, we are interested in understanding how often $f(\C)$ intersects  $D$. 

Denote by  $L$ the line bundle generated by $D$.  Let $\D_r$ be the disk of radius $r$ centered at the origin in $\C$. Based on the exhaustion $\C=\cup_{r>0}\D_r$, we will count the number of intersection points between $f(\D_r)$ and $D$, which is finite. Precisely, taking the $k$-truncated degrees of the divisor $f^*D$ on disks by
\[
n_f^{[k]}(t,D)
:=
\sum_{z\in\D_t}
\min
\,
\{k,\ord_zf^*D\}
\eqno
{{\scriptstyle (t\,>\,0)},}
\]
the \textsl{truncated counting function of $f$ at level} $k$ with respect to $D$ is then defined by taking the logarithmic average
\[
N_f^{[k]}(r,D)
\,
:=
\,
\int_1^r \frac{n_f^{[k]}(t, D)}{t}\,\d t
\eqno
{{\scriptstyle (r\,>\,1)}.}
\]
When $k=\infty$, we write $n_f(r,D)$, $N_f(r,D)$ instead of $n_f^{[\infty]}(r,D)$, $N_f^{[\infty]}(r,D)$. These functions count the number of points in $f(\D_r) \cap D$, taking into account the multiplicity. On the other hand, when $k=1$, the function $n^{[1]}_f(r,D)$ gives us the number of points in $f(\D_r) \cap D$ as a set.  We then call
 $N^{[1]}_f(r,D)$ the \emph{geometric counting function} of $f$ with respect to $D$.
 
 For every smooth $(1,1)$-form $\eta$ on $X,$ we put 
\[
T_f(r,\eta):= \int_1^r \frac{\d t}{t} \int_{\D_t} f^* \eta\eqno
{{\scriptstyle (r\,>\,1)}.}
\]
Observe that if $\eta, \eta'$ are two smooth closed $(1,1)$-forms in the same cohomology class, then $T_f(r, \eta)= T_f(r, \eta') + O(1)$ as $r \to \infty$ by the Lelong-Jensen formula. It follows that given a smooth Chern form $\omega_L$ of $L$, the \emph{characteristic function} of $f$ with respect to $L$ given by
\[
T_f(r,L):= T_f(r, \omega_L)
\]
is well-defined up to a bounded term as 
$r\to \infty$. 

By the First Main Theorem \cite[Th. 2.3.31]{Noguchi}, there holds
\[
N_f^{[1]}(r,D) \le N_f(r,D) \le  T_f(r,L) +O(1).
\]

On the other side, in Second Main Theorem, one tries to establish a lower bound for counting functions. Many such type of estimates are based on the work of Cartan \cite{Cartan1933}.  Let $X = \P^n$ and $f:\mathbb{C}\rightarrow\mathbb{P}^n$ be an entire holomorphic curve whose image is not contained in any hyperplane.  Denote by $T_f(r)$ the characteristic function of $f$ with respect to the hyperplane bundle of $\mathbb{ P}^n$. Let $\{H_j\}_{1\leq i\leq q}$ be a family of $q\geq n+2$ hyperplanes in $\mathbb{ P}^n$ in general position, \emph{i.e.,} any collection of $n+1$ members in this family has empty intersection. Then the classical Cartan's Second Main Theorem states that
\begin{align}\label{ine-cartan}
(q-n-1)T_f(r) \le \sum_{j=1}^q N^{[n]}_f(r,H_j) + O(\log T_f(r)+ \log r),
\end{align}
for $r$ outside a set of finite Lebesgue measure on $\R$.

An important  problem in Nevanlinna's theory is  to decrease the truncation level in (\ref{ine-cartan}) as low as possible. When $n=2$, the truncation level $2$ is optimal as showed by an example in \cite{Duval-Tuan}. However, it was conjectured that (\ref{ine-cartan}) still holds for $N^{[1]}_f(r,H_j)$ in place of $N^{[n]}_f(r,H_j)$, provided that $f$ is algebraically non-degenerate, \emph{i.e.}, $f(\C)$ is not contained in any proper algebraic subset of $X$. This conjecture is widely open. Note that in the context of abelian or semi-abelian varieties, such a Second Main Theorem type estimate with the optimal truncation level $1$ has been established in \cite{NWY2,NWY,Yamanoi}. The reader is referred to  \cite{Duval-Tuan,McQuillan,Noguchi,Ru} and references therein for more informations. 

Recall that the \emph{defect} of $D$ with respect to $f$ is defined as
\[
\delta_f(D):= \liminf_{r \to \infty}\bigg(1- \frac{N_f(r,D)}{T_f(r,L)}\bigg).
\]
For $k \in \N^*$, \emph{ the defect of truncation $k$} of $D$, which is denoted by  $\delta^{[k]}_f(D)$,  is then defined in a similar way with $N^{[k]}_f(r,D)$ in place of $N_f(r,D)$. We call $\delta^{[1]}_f(D)$ \emph{the geometric defect of $D$}.  It is trivial that
$$0 \le \delta_f(D) \le \delta_f^{[k]}(D)  \le \delta_f^{[1]}(D) \le 1$$
for $k \ge 1$.   A divisor with zero  geometric defect roughly signifies that  the logarithmic average growth of the cardinality of the set $f(\D_r) \cap D$ is the same as that of the area of $f(\D_r)$ as $r \to \infty.$ In the other extreme case where $\delta_f(D)=1$, the logarithmic average of the cardinality of the set $f(\D_r) \cap D$ counted with multiplicity is negligible with respect to the area of $f(\D_r)$. 

 Observe that a direct consequence of  (\ref{ine-cartan}) is that for $X= \P^n$ and for  very generic hyperplane $D$ in $\P^n$, we have $\delta^{[n]}_f(D)=0$, see Lemma \ref{le-truncationn} below.  The goal of this paper is to prove the following stronger statement which serves as an evidence supporting the above conjecture. 

\begin{theorem} \label{thedefect0}  Let $X$ be a complex projective manifold and $L$ a very ample line bundle on $X$. Let $E$ be the space of holomorphic sections of $L$.  Let $f: \C \to X$ be a transcendental holomorphic curve. Then for every effective divisor $D$ of $L$ outside a countable union of proper algebraic subsets of $\P(E)$, we have $\delta_f^{[1]}(D)=0$. 
\end{theorem}

We underline that the feature of Theorem \ref{thedefect0} is the truncation $1$ of the defect. By a recent result in \cite{brotbek-deng}, for every integer $d$  greater than an explicit number depending on $n$ and  for a generic  divisor $D$ of  $L^d$, we have $\delta^{[1]}_f(D) \le 1- 1/d$; see also \cite{Vu_Huynh_Xie} for a weaker estimate.  We also notice that Theorem \ref{thedefect0} is sharp in the case where $n=1$ by a result of Drasin \cite{Drasin}. 

Recall that every effective divisor $D$ of $L$ is the zero divisor of a holomorphic section of $L$ which is naturally identified with a complex line passing through the origin of $E.$ Hence, we can view $D$  as a point in $\P(E)$.  Note that we don't require that $f$ is algebraically non-degenerate and Theorem \ref{thedefect0} is clear if $f$ is a non-constant rational curve because the image of $f$ is an algebraic curve in $X.$  By using a basis of $E,$ we can embed $X$ into a projective space and the problem can be reduced to the case where $X= \P^n$ and $L$ is the hyperplane bundle. But this reduction doesn't make the problem easier.

Consider now an analytic subset $V$ of $X.$ By using local basis of the sheaf of ideals of holomorphic functions defining $V$, we can also define the counting function $N_f(r,V)$ even if $\codim V \ge 2$, see \cite[Sec. 2.4.1]{Noguchi} for details. Here is our second main result.   

\begin{theorem} \label{the-codim2V} Let $Z$ be a complex manifold.  Let $X$ be a compact K\"ahler manifold and $\omega$ a K\"ahler form on $X$. Let $f: \C \to X$  a transcendental holomorphic curve. Let $\mathcal{V}$ be a complex smooth submanifold of $X\times Z$ of codimension $s\ge 2$. Denote by $p_1,p_2$ the natural projections from $X\times Z$ to $X$, $Z$ respectively.  Assume that  the restriction $p_{1,\mathcal{V}}$ of $p_1$ to $\mathcal{V}$ is a submersion and the restriction $p_{2,\mathcal{V}}$ of $p_2$ to $\mathcal{V}$ is a surjection.  Then for any $a$ outside a countable union of proper analytic subsets of $Z$, we have 
\begin{align}\label{eq-the-codim2}
\liminf_{r\to \infty} \frac{N_f(r,\mathcal{V}_a)}{T_f(r,\omega)}=0,
\end{align}
 where  $\mathcal{V}_a:= p_1\big(\mathcal{V} \cap (X\times \{a\})\big)$. 
\end{theorem}
 
Observe that the fiber of $p_{2,\mathcal{V}}$ at a point $a \in Z$ is $\mathcal{V}_a$. Since $p_2$ is surjective and proper, the set of critical values of $p_2$ is a proper analytic subset of $Z$ and  for every $a$ outside this set,  we see that $\mathcal{V}_a$ is of codimension $s$ in $X$ and $\mathcal{V}$ is transverse to $X \times \{a\}.$ Hence $\mathcal{V}$ is essentially a family of analytic subsets of codimension $s\ge 2$ in $X$. The above result roughly says that for an analytic subset $V$ of codimension $2$ general enough,  then $\liminf_{r\to \infty} \frac{N_f(r,V)}{T_f(r,\omega)}=0$, \emph{i.e.,} $f$ almost misses such $V$.

A simple example to which we can apply Theorem \ref{the-codim2V} is when $X= \P^n$, $Z$ is the space of projective subspaces of codimension $s\ge 2$ of $X$ and $\mathcal{V}$ is the family of subspaces of codimension $s$ of $\P^n$ which is viewed as a submanifold of $\P^n \times Z$.  Even in this situation, it seems that the conclusion of the above theorem is still new. Note that when $X$ is an abelian or semi-abelian variety, by the results of Noguchi-Winkelmann-Yamanoi \cite{NWY2,NWY,Yamanoi}, one has $\liminf_{r\to \infty} \frac{N_f(r,V)}{T_f(r,\omega)}=0$ for every analytic set $V$ of codimension $\ge 2$ in $X$. 

We would like to make some comments about our strategy to prove main results.  The proof of Theorem \ref{the-codim2V} consists of two main steps. In the first step,  we look at a Nevanlinna's $d$-closed current $S$ of bidimension $(1,1)$ associated to $f$ and study the intersection of $S$ with $\mathcal{V}_a$. Since $S$ is of bi-dimension $(1,1)$ and $\mathcal{V}_a$ is of codimension at least $2$ in $X$,  one is tempted to expect that the formal intersection $``S \wedge \mathcal{V}_a"$ should be zero. Here, the crucial point is that one needs to interpret this intersection in a proper sense. In our proof, this will be done by using the theory of density currents coined recently by Dinh-Sibony in \cite{Dinh_Sibony_density}. Roughly speaking, we will prove that for very generic $a \in Z$ (\emph{i.e,} for $a$ outside a countable union of proper analytic subsets of $Z$), modulo natural identifications, one has 
\begin{align}\label{eq-ideathe1.3}
``S \wedge [\mathcal{V}_a]"= `` p_1^* S \wedge [\mathcal{V}] \wedge [X \times \{a\}]" =``\big(p_1^* S \wedge [\mathcal{V}]\big) \wedge [X \times \{a\}]"=0,
\end{align}
where for an analytic set $V$, we denote by $[V]$ the current of integration along $V$.  Here the intersection of currents in the last equalities will be understood in the sense of density currents. 

 The first equality of (\ref{eq-ideathe1.3}) follows from  the fact that $\mathcal{V}_a$ is naturally identified with  $\mathcal{V} \cap (X \times \{a\})$. The second one comes from a general fact about the associativity of density currents (see Theorem \ref{the-densitybangkhong}). Intuitively, the last equality of  (\ref{eq-ideathe1.3})  is true due to the fact that the bi-degree of the current $``\big(p_1^* S \wedge [\mathcal{V}]\big) \wedge [X \times \{a\}]"$ is strictly greater than the dimension of the ambient space $X \times Z$, thanks to the assumption that $\mathcal{V}$ is of codimension at least $2$. Nevertheless, to make the last argument rigorously, we need to prove  a result concerning the intersection of a current with generic fibers of a submersion which is of independent interest (see Theorem \ref{th-slice-density}).

In the second step, we will finish the proof by  relating  the current $``S \wedge [\mathcal{V}_a]"$ with the right-hand side of  (\ref{eq-the-codim2}). This will be done by considering  the blowup of $X$ along $\mathcal{V}_a$. Note that in this step,  we need to work with $\ddc$-closed (not necessarily $d$-closed) Nevanlinna's current associated to the lift of $f$ to the last blowup.

The proof of Theorem \ref{thedefect0} goes first by lifting $f$, $D$ to a curve $\widehat f$ and an analytic subset $\widehat D$ of codimension $2$ in the projectivisation $\P(\Ta X)$ of the tangent bundle $\Ta X$ of $X$.  The intersection of $\widehat f$ and $\widehat D$ encodes the points where $f$ is tangent to $D$. This is an idea inspired by \cite{DS_henon,D_suite}. The characteristic function of $\widehat f$ and that of $f$ are comparable by the lemma on logarithmic derivative or McQuillan's tautological inequality.  This fact permits us to use the characteristic function of $\widehat f$ instead of that of $f$. 

When $D$ runs over the set  $\P(E)$  of divisors, $\widehat D$ forms a smooth family $\mathcal{V}$ of analytic  subsets of codimension $2$ in $\P(\Ta X)$. The last family can be seen as a smooth submanifold  of codimension $2$ in $\P(\Ta X) \times \P(E)$.  This allows us to use Theorem \ref{the-codim2V} for $\widehat f$ and $\mathcal{V}$ to obtain $``\widehat S \wedge [\widehat D]" =0$ for very generic $\widehat D$. This is a key in our proof. This property together with some direct computations on the blowup of $\P(\Ta X)$ along $\widehat D$ gives the desired equality.  
  
The paper is organized as follows. In Section \ref{sec-slicing}, we present preparatory results about density currents which will be needed for our proofs of main theorems. We prove  Theorems  \ref{the-codim2V} and  \ref{thedefect0}  in Sections \ref{sec-trunc1} and \ref{sec-thjedefecto} respectively.
\\

\noindent
\textbf{Acknowledgments.}  We would like to thank Tien-Cuong Dinh, Si Duc Quang and Song-Yan Xie for fruitful discussions.  We also want to express our gratitude to the referee for comments which improved the presentation of the paper. D. T. Huynh is grateful to the Max Planck Institute for Mathematics in Bonn for its hospitality and financial support. He also wants to acknowledge the support from Hue University.  D.-V. Vu  is supported by a postdoctoral fellowship of  the Alexander von Humboldt Foundation.

\section{Density currents} \label{sec-slicing}

In this section, we will prove several results concerning density currents which are crucial for our proofs of  main results.  Firstly, we need to recall some known facts about the density currents proved in  \cite{Dinh_Sibony_density},  see also \cite{Vu_density-nonkahler} for some simplifications.  

Let $Y$ be a compact K\"ahler manifold.  Let $V$ be  a smooth submanifold of dimension $\ell$ of $Y$. Let $[V]$ be the current of integration along $V$.   Let $T$ be a positive closed current of bi-degree $(p,p)$ in $Y$. Assume that $T$ has no mass on $V$. We denote by $\{T\}, \{V\}$ the cohomology class of $T, [V]$ respectively. Denote by $\pi: E\to V$ the normal bundle of $V$ in $Y$ and $\overline E:= \P(E \oplus \C)$ the projective compactification of $E$. The hypersurface at infinity $H_\infty: = \overline E \backslash E$ of $\overline E$  is naturally isomorphic to $\P(E)$ as fiber bundles over $V$. We also have a canonical projection $\pi_\infty: \overline E \backslash V  \to H_\infty $. 


A smooth diffeomorphism $\tau$ from an open subset $U$ of $Y$  to an open neighborhood of $V$ in $E$ is called \emph{admissible} if  $\tau$ is the identity map on $V\cap U$  and  the restriction of its differential $\d\tau$ to $E|_{V\cap U}$ is the identity map.   By \cite[Le. 4.2]{Dinh_Sibony_density}, there exists an admissible map $\tau: U \to E$ such that $U$ is  a  tubular neighborhood  of $V$. In general, such a $\tau$ is not holomorphic. 

For $\lambda \in \C^*,$ let $A_\lambda: E \to E$ be the multiplication by $\lambda$ on fibers of $E$. Let $\tau$ be an admissible map defining on a tubular neighborhood of $V$. By \cite{Dinh_Sibony_density}, the family of closed currents $(A_\lambda)_* \tau_* T$ has a uniformly bounded mass on compact sets of $E$. Any limit current of this family in $E$ is called \emph{a tangent current to $T$ along $V$}.  Such a  current is a positive closed current  invariant by $A_\lambda$ and can be extended to a current in $\overline E$. Tangent currents are independent of $\tau$, this means that  if $R:= \lim_{n\to \infty}(A_{\lambda_n})_* \tau_* T$ is a tangent current, then for every admissible map $\tau': U' \to E$, we also have $R= \lim_{n\to \infty} (A_{\lambda_n})_* \tau'_* T$ on $\pi^{-1}(U'\cap V)$. This property allows us crucial flexibility in choosing   admissible maps when we need to estimate tangent currents in practice.  
Thus if we work locally, then after trivializing $E$,  the admissible map in the definition of tangent currents  can  be chosen to be the identity. 

 In general, the tangent currents to $T$ along $V$ are not unique but their cohomology classes in $\overline E$ are unique.  That unique class is denote by $\kappa^V(T)$ and called \emph{the total tangent class} to $T$ along $V$.   Let $h_{\overline E}$ be the Chern class  of  the dual of the tautological line bundle of $\overline E$ respectively. We can write $\kappa^V(T)$ uniquely as 
\begin{align} \label{eq_kappaVT}
\kappa^V(T)=\sum_{j=\max \{0, \ell-p\}}^{\min\{\ell, k-p-1\}} \pi^* \kappa_j \wedge h_{\overline E}^{p-(\ell-j)},
\end{align} 
where $\kappa_j$ is a cohomology class in $H^{\ell-j,\ell-j}(V)$, which is called the $j^{th}$ component of $\kappa^V(T)$.

Let $T_1, \ldots, T_m$ be positive closed currents on $Y$. Let $T_1 \otimes \cdots \otimes T_m$ be the tensor current of $T_1, \ldots, T_m$ on $Y^m$.  A \emph{density current} associated to $T_1, \ldots, T_m$ is a tangent current of $T_1 \otimes \cdots \otimes T_m$ along the diagonal $\Delta_m:= \{(y,\ldots,y): y\in Y\}$ of $Y^m$. If there is a unique density current $T_\infty$ to $T_1, \ldots, T_m$  and  $T_\infty= \pi_m^* S$ for some positive closed current $S$ on $\Delta_m$, where $\pi_m:E_{\Delta_m}\rightarrow \Delta_m$ is the projection from the normal bundle $E_{\Delta_m}$, then we say that the \emph{Dinh-Sibony product} $T_1 \curlywedge \cdots \curlywedge T_m$ is well-defined  and put $T_1 \curlywedge \cdots \curlywedge T_m:= S$.

Consider a particular case where  $T_1:= T$ and $T_2:= [V]$. Observing that we have  natural identifications $\Ta (Y^2) \approx \Ta Y \times \Ta Y$ between vector bundles, where $\Ta Y$ is the tangent bundle of $Y$ and $\Delta \approx Y$.  Since $V \subset Y\approx \Delta$, there is a canonical inclusion $\imath$ from  $\Ta V$ to $(\Ta Y \times \{0\})|_\Delta$ which is a subbundle of $\Ta (Y^2)|_{\Delta}$. Let $F$  be the image of $\imath(\Ta V)$ in the normal bundle $E_\Delta= \Ta (Y^2)/ \Ta \Delta$. Put $\Delta_V:= \{(y,y) \in Y^2:y \in V\}$.  Let $E_{\Delta, V}$ be the restriction of $E_{\Delta}$ to $\Delta_V$.  Observing that $F$  is a subbundle of $E_{\Delta,V}$ of rank $\ell$ and the natural map 
$$\Psi: E_{\Delta, V}/F \to E=\Ta Y/ \Ta V$$
 is an isomorphism. Let $p_V: E_{\Delta,V} \to E_{\Delta,V}/F $ be the natural projection. The following result tells us  that a density current associated to $T$, $[V]$ corresponds naturally to a tangent current of $T$ along $V$. 

\begin{lemma} (\cite[Le. 5.4]{Dinh_Sibony_density} or \cite[Le. 2.3]{Vu_density-nonkahler}) \label{le_deltatangecurr}  If $T_\infty$ is  a tangent current of $T$ along $V,$ then the current $p_V^* \Psi^* T_\infty$ is a tangent current of $T \otimes [V]$ along $\Delta$.  Conversely, every tangent current of $T\otimes [V]$ along $\Delta$ can be written as $p_V^* \Psi^* T_\infty$ for some tangent current $T_\infty$ of $T$ along $V$. 
\end{lemma}

We present below some more properties of tangent currents which will be used in the sequel. Let $\sigma: \widehat Y \to Y$ be the blowup along $V$ of $Y$ and $\widehat V:= \sigma^{-1}(V)$ the exceptional hypersurface. Recall that $\widehat V$ is naturally biholomorphic to $\P(E)$.  

Let  $\sigma_E: \widehat E \to E$ be the blowup along $V$ of $E$. The projection $\pi$ induces naturally a vector bundle projection $\pi_{\widehat E}$ from $\widehat E$ to $\sigma_E^{-1}(V)$. Let $\overline{\widehat E}$ be the projective compactification of the vector bundle $\widehat E$.  The  map $\pi_{\widehat E}$ can be extended to a projection $\pi_{\overline{\widehat E}}:\overline{\widehat E}\to \sigma_E^{-1}(V)$. We also denote by $\sigma_E$ its natural extension from $\overline{\widehat E} \to \overline{E}$.
 The vector bundle $\pi_{\overline{\widehat E}}: \widehat E\to \sigma_E^{-1}(V)$ is naturally identified with the normal bundle of $\widehat V$ in $\widehat Y$.  Hence we can identify $\sigma_E^{-1}(V)$ with $\widehat V$ and use $\widehat E$ as the normal bundle of $\widehat V$ in $\widehat Y$.  

Recall that since $Y$ is K\"ahler,  so is $\widehat Y$.  If $\codim V \ge 2,$ let $\widehat\omega_h$ be a smooth Chern form of the line bundle $\mathcal{O}(-\widehat V)$ whose restriction to each fiber of  $\widehat{V} \approx \P(E)$ is strictly positive, otherwise we simply put $\widehat \omega_h:= 0$. By rescaling $\omega$ if necessary, we can assume that  $\widehat \omega:=  \sigma^* \omega+ \widehat \omega_h>0$.

Observe that the hypersurface at infinity $\widehat H_\infty$ of  $\overline{\widehat E}$ is biholomorphic to  that of $\overline E$ via $\sigma_E$. We  use $\widehat{\pi}_\infty$ to denote the natural projection from $\overline{\widehat E}\backslash \widehat V$ to $\widehat H_\infty.$ Since the rank of $\overline{\widehat E}$ over  $\widehat V$ is $1$, we can extend $\widehat{\pi}_\infty$ to a projection from $\overline{\widehat E}$ to $\widehat H_\infty$. Thus, $\widehat V$ is naturally identified with $\widehat H_\infty$ which is in turn naturally identified with $H_\infty$.

  Let $\widehat T$ be the pull-back of $T$ on $\widehat Y  \backslash \widehat V$ by $\sigma|_{\widehat Y \backslash \widehat V}.$ The mass of $\widehat T$ is finite by \cite{DinhSibony_pullback}. We thus can extend $\widehat T$ trivially through $\widehat V$ to a current on $\widehat Y$.

For any positive current $S$ on $H_\infty$, the positive current $\pi_\infty^* S$ has a finite mass on $\overline E \backslash V$. Hence we can extend it trivially through $V$. Denote also by $\pi_\infty^* S$ this extension.  Since 
$$\pi_\infty= \widehat \pi_\infty \circ (\sigma_E|_{\widehat E \backslash \widehat V})^{-1},$$ we can check that  
$$\pi_\infty^*S= (\sigma_E)_* \circ (\widehat \pi_\infty)^*S.$$
By the last formula, the map $\pi_\infty^*$ induces natural maps on the cohomology groups and $\pi_\infty^*$ is continuous. For a cohomology class $\alpha$ in $X$, we denote by $\alpha|_V$ the restriction of  $\alpha$ to $V$. 

 Let $T_\infty$ be a tangent current to $T$ along $V$. We can check directly that there exists a tangent current   $\widehat T_\infty$    to $\widehat T$  along $\widehat V$ satisfying
\begin{align} \label{eq_TifntysigmaE}
T_\infty= (\sigma_E)_* \widehat T_\infty.
\end{align}
We have the following important property. 

\begin{proposition}\label{pro_Sinfty}(\cite{Dinh_Sibony_density}) The current $T_\infty$ is $V$-conic, \emph{i.e.,} $A_\lambda^* T_\infty= T_\infty$ for every $\lambda \in \C^*$. Equivalently,  there exists a positive closed current $S_\infty$ on $H_\infty$ such that 
\begin{align}\label{eq_TifntysigmaE22}
T_\infty= \pi^*_\infty S_\infty.
\end{align}
Moreover, we have
\begin{align} \label{eq-cohoclassoftangentcurrent}
\kappa^V(T)= \pi^*_\infty\{S_\infty\}= \pi^*_\infty(\{\widehat T_\infty\} \smile \{\widehat V\}),
\end{align}
where recall that  we identified $\widehat V$ with $\widehat H_\infty$ and $\widehat H_\infty$ with $H_\infty$.
\end{proposition}

We are going to present the first main result of this section concerning the density currents associated to a current and slices of a given submersion.  Recall that the mass of a current $S$ of order $0$ on a manifold $\mathcal{M}$ is given by 
\[
\|S\|_{\mathcal{M}}:=\sup_{\|\Phi\|_{\cali{C}^0}\le 1} |\langle S, \Phi \rangle|,
\]
where the supremum is taken over all  smooth differential forms $\Phi$ on $\mathcal{M}$. From now on, the notations $\lesssim, \gtrsim$ are used to indicate $\le, \ge$ modulo a multiplicative constant, respectively.   

Let $W$ be a compact K\"ahler manifold and $\pi_W: Y\to W$ a holomorphic submersion. Put $Y_\theta:=\pi_W^{-1}(\theta)$ for every $\theta \in W$. Let $E_\theta$ be the normal bundle of $Y_\theta$ in $Y$ and $\pi_{\P(E_\theta)}:\P(E_\theta)\rightarrow Y_\theta$ be the natural projection. Let $m:= \dim W$ and  $\ell:= \dim Y_\theta$, which is independent of $\theta$. 

\begin{theorem} \label{th-slice-density} There exists a subset $\mathcal{A}$ of $W$ which is a countable union of proper analytic subsets of $W$ such that for any $\theta \notin \mathcal{A}$, the component $\kappa^{Y_\theta}_j(T)$ of the total  tangent class to $T$ along  $Y_\theta$ is zero for any $j> \ell- p$. In particular, if $p> \ell,$ then the tangent current to $T$ along $Y_\theta$ is zero, or equivalently  $T\curlywedge [Y_\theta]=0$.
\end{theorem}

\proof   Let $\omega$ be a K\"ahler form on $Y$.  For $j \ge  \max\{\ell-p,0\},$ the current 
$$T_j:=(\pi_W)_*(T \wedge \omega^j)$$
is positive closed of bidegree $(s,s)$ in $W$, where  $s:=p+j-\ell$. If $j>\ell-p$, then $s\ge 1$. Let $\mathcal{A}_j$ be the set of $\theta \in W$ such that $\nu(T_j, \theta)>0$. Observing that $\mathcal{A}_j$ is a countable union of proper analytic subsets of $W$ by Siu's semi-continuity theorem if $j>\ell-p$. Set
\[
\mathcal{A}:=\bigcup_{j \ge  \max\{\ell - p+1,0\}} \mathcal{A}_j.
\]
Let $\theta_0 \not \in \mathcal{A}.$  Thus for any $j\geq \ell-p+1$, one has $\nu(T_j, \theta_0)=0.$ 

Consider a tangent current $R$ to $T$ along $Y_{\theta_0},$ \emph{i.e.,}  $R=\lim_{n\to \infty}(A_{\lambda_n})_* \tau_* T,$ where $\tau$ is an admissible map from a tubular neighborhood of $Y_{\theta_0}$ to $E_{\theta_0}$. Let $j_0$ be the maximal $j$ such that $\kappa^{Y_\theta}_j(T) \not =0$. Note that by a bi-degree reason, we have $j_0 \ge \max\{0, \ell-p\}$. 

 Suppose that $j_0>\ell-p$ because otherwise we have nothing to prove.  Recall that $\kappa^{Y_{\theta_0}}_j(T)$ is a class in $H^{\ell-j,\ell-j}(Y_{\theta_0})$.  It follows that the mass of the positive closed current  $R_{j_0}:=R \wedge \pi_{\P(E_\theta)}^* \omega_{\theta_0}^{j_0}$ is strictly positive, where  $\omega_\theta$  is the restriction of $\omega$ to $Y_\theta$. Moreover since $R$ is $Y_\theta$-conic, so is $R_{j_0}.$ Thus, the mass of  $R_{j_0}$ on $\overline E_\theta$ is bounded by a constant times the mass of  $R_{j_0}$ on every open tubular neighborhood of $Y_{\theta_0}$, see \cite[Le. 3.16]{Dinh_Sibony_density}. Hence there exists a local chart $U$ on $Y$ such that 
\begin{align} \label{eq-tangenthdimensionjT}
\|R_{j_0}\|_U>0,
\end{align}
where we identified $U$ with a local chart of $E_{\theta_0}$.

From now on we work locally on $U$ and without loss of generality, we can assume that  $U= U_1 \times U_2$ of $Y$ and $(x,\theta)$ the coordinate system on $U$ such that $Y_{\theta_0}\cap U= \{\theta=0\}$. Hence $U_2$ is a local chart of $W$ centered at $\theta_0$ and $x$ is a local coordinate system on $Y_\theta$. We can identify $E_\theta$ with $U_1 \times \C^{m}$ and $A_\lambda$ is given by the multiplication $(x, \theta) \mapsto (x,\lambda \theta)$. Denote by $A'_{\lambda}$ the multiplication $\theta \mapsto \lambda \theta$ in $\C^{m}$, one sees that $A'_{\lambda} \circ \pi_W= \pi_W \circ A_{\lambda}$. 

The natural projection $\pi_{U_1}:U_1 \times U_2\rightarrow U_1$ is the restriction of $\pi_{\P(E_\theta)}$ to $\pi_{\P(E_\theta)}^{-1}(U_1)$. 
Denote by $\id_U$ the identity map on $U$. Observe that $\id_U$ is an admissible map by our identification of $E_\theta$ with $U_1 \times \C^{m}$. On $U_1 \times \C^{m}$,  we get
\begin{align}\label{eq-RtauphayU}
R=\lim_{n\to \infty}(A_{\lambda_n})_* (\id_U)_* T=\lim_{n\to \infty}(A_{\lambda_n})_* T.
\end{align}
For each $j>\ell-p$, recall that $T_j$ is a current on $W$ of bi-degree $(s,s)$ with $s\geq 1$. We then have 
\begin{align}
\label{eq-chieupiWTomegaj_0}
\nu(T_j, \theta_0) &\gtrsim \limsup_{\lambda \to \infty} \|(A'_{\lambda})_*  T_j \|_{U_2}=\limsup_{\lambda \to \infty} \|(A'_{\lambda})_* (\pi_W)_*(T \wedge \omega^j) \|_{U_2}\notag \\
&=\limsup_{\lambda \to \infty} \|(\pi_W)_*\big((A_\lambda)_* T \wedge  \pi_{U_1}^* \omega_{\theta_0}^j\big)\|_{U_2}.
\end{align} 
 Let $\omega_{W}$ be a K\"ahler form on $W.$  Observe that   $\omega \lesssim \omega_{\theta_0}+ \omega_{W}.$ It follows that  for  $r:= \dim Y-(p+ j_0)$, one has  
$$\limsup_{\lambda \to \infty} \big \langle (A_\lambda)_* T \wedge  \pi_{U_1}^* \omega_{\theta_0}^{j_0}, \omega^r  \big \rangle \lesssim \limsup_{\lambda \to \infty}\sum_{r'=0}^r   \|(\pi_W)_*\big[ (A_\lambda)_* T \wedge  \pi_{U_1}^* \omega_{\theta_0}^{j_0+r'}\big]\|\lesssim \sum_{r'=0}^r   \nu(T_{j_0+r'}, \theta_0)=0$$
by (\ref{eq-chieupiWTomegaj_0}) and our choice of $\theta_0$.   
This together with (\ref{eq-RtauphayU}) and (\ref{eq-tangenthdimensionjT}) gives a contradiction. Hence $j_0 \le \ell-p$. It follows that $j_0= \ell-p$. We also deduce that if $p>\ell,$ then the tangent current $R$ is zero. This finishes the proof.
\endproof

The next two results concerns the associativity of density currents: given currents $T,R,S$ such that $T\curlywedge R$ is well-defined and $(T \curlywedge R) \curlywedge S$ is well-defined, is  $T\curlywedge R \curlywedge S$ well-defined and  equal to $(T \curlywedge R) \curlywedge S$? We are not able to answer this question in general but we can show that this is the case in some situations which are enough for our applications later.  

\begin{theorem}\label{the-densitybangkhong} Let $R_1, R_2$ be  positive closed currents of bi-degree $(p_1, p_1)$,  $(p_2,p_2)$ respectively on $Y$ such that the Dinh-Sibony product $T:= R_1 \curlywedge R_2$ is well-defined. Assume that $p:=p_1+p_2> \ell$. 
Then there exists a subset $\mathcal{A}$ of $W$ which is a countable union of proper analytic subsets of $W$ such that for any $\theta \notin \mathcal{A}$,  we have $R_1 \curlywedge R_2 \curlywedge  [Y_\theta]=0$. 
\end{theorem}

\proof  Note that we proved in Theorem \ref{th-slice-density} that $(R_1 \curlywedge R_2) \curlywedge [Y_\theta]= T \curlywedge [Y_\theta]=0$ for $\theta$ outside a countable union of proper analytic subsets of $Z$. Our desired assertion doesn't follow directly from the last property because in general we don't know whether $R_1 \curlywedge R_2 \curlywedge [Y_\theta]$ is well-defined and  equal to $(R_1 \curlywedge R_2) \curlywedge [Y_\theta]$.

Let $\Delta_3$ be the diagonal of $Y^3$ which is the set of $(y,y,y)$ for $y\in Y$. Let $E_3$ be the normal bundle of $\Delta_3$ in $Y^3$. Let $A_{\lambda,3}$ be the  multiplication by $\lambda$ along fibers of $E_3$. Let $\Delta$ be the diagonal of $Y^2$ and $E$ the normal bundle of $\Delta$ in $Y^2$. Let $A_{\lambda}$ be the multiplication by $\lambda$ along fibers of $E$. Let $\pi_E$ be the projection from $E$ to $\Delta$. 
   
Let $T_j, \mathcal{A}_j$ be the currents and the set  given in the proof of Theorem \ref{th-slice-density} for $j\ge 0$. Since $p>\ell$, the set $\mathcal{A}_j$ is a countable union of proper analytic subsets of $W$.  
 The desired assertion is a direct consequence of the following inequality:
\begin{align}\label{ine-lelonglimit2}
\limsup_{\lambda \to \infty} \|(A_{\lambda,3})_*(R_1 \otimes R_2 \otimes [Y_\theta]\| \lesssim  \sum_{j\ge 0}\nu(T_j, \theta) 
\end{align}
for $\theta \in W$. Indeed, if $\theta \not \in \mathcal{A}:= \cup_{j\ge 0} \mathcal{A}_j$, then the right-hand side of  the above inequality is zero, which implies that the density current associated to $R_1, R_2, Y_\theta$ is zero.

Let us now  prove the inequality (\ref{ine-lelonglimit2}). Let $\theta_0\in W$.  From now on, we work locally near $Y_{\theta_0}$. Let $U=U_1 \times U_2$ be a local chart near $Y_{\theta_0}$  with the coordinates $y=(x,\theta)$ such that $\theta_0=0$ and $\pi_W(x,\theta)=\theta$. We obtain induced coordinates $(y,y',y'')$ on $U^3 \subset Y^3,$ where $y'=(x',\theta'), y''=(x'',\theta'')$. Put 
$$\tilde{y}'=(\tilde{x}',\tilde{\theta}'):= y' -y, \quad \tilde{y}''=(\tilde{x}'',\tilde{\theta}''):= y''-y.$$
Then $(y,\tilde{y}',\tilde{y}'')$ are new local coordinates on $U^3$ and 
\[
\Delta_3=\{\tilde{y}'=\tilde{y}''=0\}, \quad \Delta= \{\tilde{y}'=0\}.
\]
Identify $E_3$ over $\Delta_3 \cap U^3$ with $U \times \C^{2 \dim Y}$ and $E$ with $U \times \C^{\dim Y}$. The multiplication $A_{\lambda,3}$ is given by $(y,\tilde{y}', \tilde{y}'') \mapsto (y,\lambda\tilde{y}', \lambda\tilde{y}'')$ and $A_{\lambda}$ is given by  $(y,\tilde{y}') \mapsto (y,\lambda\tilde{y}')$. Let  $\Phi(y, \tilde{y}', \tilde{y}'')$
be a positive test form with compact support. Let $\tau_3$ be the change of coordinates from $(y,y',y'')$ to $(y,\tilde{y}',\tilde{y}'')$ and $\tau$ the change of coordinates from $(y,y')$ to $(y, \tilde{y}')$.  Put
\[
Q_\lambda:= \big\langle (A_{\lambda,3})_*(\tau_3)_*(R_1 \otimes R_2 \otimes [Y_{\theta_0}]), \Phi(y, \tilde{y}', \tilde{y}'') \big\rangle,
\]
which is equal to 
\[\big \langle  \tau_*(R_1 \otimes R_2), \int_{\{x''\in Y_{\theta_0}\}} \Phi(y, \lambda \tilde{y}', \lambda(x''- x), - \lambda \theta)\big \rangle =\big \langle  \tau_*(R_1 \otimes R_2), \int_{\{\tilde{x}''\in \C^\ell\}} \Phi(y, \lambda \tilde{y}', \tilde{x}'', - \lambda \theta)\big \rangle.
\]
It follows that 
$$Q_\lambda= \big \langle R_\lambda,  \Phi_\lambda\big \rangle,$$
where
$$R_\lambda:= (A_\lambda)_* \tau_*(R_1 \otimes R_2), \quad   \Phi_\lambda(y, \tilde{y}', \theta):= \int_{\{\tilde{x}''\in \C^\ell\}} \Phi(y,\tilde{y}', \tilde{x}'', - \lambda \theta)\big \rangle.$$
Observe that 
$$ \Phi_\lambda(y, \tilde{y}', \theta) \lesssim \sum_{j=0}^{q} \Omega_j(x, \tilde{y}') \wedge \omega_{W}^j(\lambda \theta),$$
for some positive integer $q$,  where $\omega_W$ is a K\"ahler form on $W$ and $\Omega_j$ are positive test forms with compact supports. Put $R_{\lambda,j}:= (\pi_W)_* (\pi_E)_*\big(R_\lambda \wedge  \Omega_j(x, \tilde{y}')\big)$.   This implies that 
\begin{align}\label{ine-uocluongQ;ambdaRlambda}
Q_\lambda \lesssim  \big \langle R_\lambda,  \Omega_j(x, \tilde{y}') \wedge \omega_{W}^j(\lambda \theta) \big\rangle=\big \langle R_{\lambda, j}, \omega_{W}^j(\lambda \theta) \big\rangle,
\end{align}
recall here that we identified $\Delta$ with $Y$ and the bracket is computed over $U_2$. By hypothesis that  
$$\lim_{\lambda \to \infty} R_\lambda= \pi_E^* (R_1 \curlywedge R_2)= \pi_E^* T,$$ we get 
\begin{align}\label{eq-limitRlamnbdaj}
\lim_{\lambda \to \infty}R_{\lambda,j}=   (\pi_W)_*(T \wedge   (\pi_E)_* \Omega_j) . 
\end{align}

For any positive closed current $S$ of bi-dimension $(j,j)$ on $U_2$ and   every constant $\epsilon>0$, put 
\[
\nu(S,\theta_0,\epsilon):=\epsilon^{-2j}\langle S, \bold{1}_{\{|\theta|\le \epsilon\}} \omega_W^j \rangle,
\]
where $\bold{1}_B$ denotes the characteristic function of a set $B$. Since we are working on $U_2$, we can take $\omega_W$ to be the standard K\"ahler form on $\C^{m}$. The function $\nu(S,\theta_0, \epsilon)$ decreases to the Lelong number $\nu(S, \theta_0)$ as  $\epsilon\to 0$.  Let $\epsilon_0$ be a strictly positive constant. Without loss of generality, we can assume that $U_2$ is contained in the unit ball in $\C^{m}$.  A direct computation shows that
\[
\big \langle R_{\lambda,j}, \omega_{W}^j(\lambda \theta) \big\rangle \le c\,  \nu( R_{\lambda,j},\theta_0, c |\lambda|^{-1}) \le c\,  \nu( R_{\lambda,j},\theta_0,  \epsilon_0)
\]
for $|\lambda|$ big enough and  some constant $c>0$ independent of $\lambda$. This combined with (\ref{eq-limitRlamnbdaj}) and (\ref{ine-uocluongQ;ambdaRlambda}) yields that  
$$\limsup_{\lambda \to \infty}Q_\lambda \lesssim  \sum_{j=0}^{m}\limsup_{\lambda \to \infty}  \nu( R_{\lambda,j},\theta_0, \epsilon_0) \le \sum_{j=0}^{m}   \nu\big((\pi_W)_*(T \wedge   (\pi_E)_* \Omega_j),\theta_0, 2\epsilon_0\big)$$
for every $\epsilon_0 >0$. Letting $\epsilon_0 \to 0$ in the last inequality gives 
$$\limsup_{\lambda \to \infty}Q_\lambda \lesssim \sum_{j\ge 0} \nu(T_j, \theta_0).$$
So (\ref{ine-lelonglimit2}) follows. This finishes the proof. 
\endproof

\begin{lemma} \label{le-sosanhRSvoiTRSdensity} Let $Y$ be a compact K\"ahler manifold and $T$ a positive closed current on $Y$. Let $V_1, V_2$ be two smooth complex submanifolds of $Y$. Assume that $V_1$ is transverse to $V_2$. If $T \curlywedge [V_1] \curlywedge [V_2]=0$, then $T \curlywedge [V_1 \cap V_2]=0$.
\end{lemma}

\proof  Let $\Delta_3=\{(x,x,x):x\in Y\}$ be the diagonal of $Y^3$.  
Let $E_3$ be the normal bundle of $\Delta_3$ in $Y^3$.  
We work locally. Let $x=(x_1,x_2,x')$ be a local coordinate system on a local chart $U= U_1 \times U_2 \times U'$ of  $Y$ such that $V_j \cap U= \{x_j=0\}$ for $j=1,2$.  We obtain induced coordinates $(x,y,z)$ on $U^3 \subset Y^3,$ where $y=(y_1, y_2, y'),$ $z=(z_1, z_2, z')$. 

Let $\ell_j= \dim V_j$ for $j=1,2$ and $k:= \dim Y$. Put 
$$\tilde{x}=(\tilde{x}_1, \tilde{x}_2,\tilde{x}'):= x -z,\quad \tilde{y}=(\tilde{y}_1, \tilde{y}_2,\tilde{y}'):= y-z.$$ 
 Observe that $(\tilde{x}, \tilde{y},z)$ are new local coordinates on $U^3$ and $\Delta_3$ is given by the equations $\tilde{x}=\tilde{y}=0$. Identify $E_3$ over $\Delta_3 \cap U^3$ with $U \times \C^{2k}$. Let $\tau$ be the identity map from $U^3 \to U \times \C^{2k}$. The fiberwise multiplication $A_\lambda$ by $\lambda$ on the normal bundle of $\Delta_3$ in $Y^3$ is given by $(\tilde{x}, \tilde{y},z) \mapsto (\lambda\tilde{x}, \lambda\tilde{y},z)$.  On the other hand, the normal bundle of $V_1 \cap V_2$ in $Y$ can be identified with $U' \times \C^{\ell_1+\ell_2-k}$ and   the fiberwise multiplication $\tilde{A}_\lambda$ by $\lambda$ on the normal bundle of $V_1 \cap V_2$ in $Y$ is given by $(x_1, x_2, x') \mapsto (\lambda x_1, \lambda x_2, x')$.

Let $\Phi(\tilde{x}, \tilde{y},z)$ be a smooth form with compact support on $U \times \C^{2k}$.  We have 
\begin{align*}
\big \langle (A_{\lambda})_* \tau_* ([V_1] \otimes [V_2]\otimes T), \Phi\big \rangle &=\big \langle [V_1] \otimes [V_2] \otimes T, \Phi\big(\lambda \tilde{x}, \lambda \tilde{y},z\big)\big \rangle \\
&= \big\langle T(z), \int_{(x_2, x') \in V_1} \int_{(y_1,y') \in V_2}\Phi(-\lambda z_1, \lambda \tilde{x}_2, \lambda \tilde{x}', \lambda \tilde{y}_1, -\lambda z_2, \lambda \tilde{y}', z) \big \rangle\\ 
&= \big \langle T(z),  \int_{\tilde{x}_2, \tilde{x}', \tilde{y}_1, \tilde{y}'} \Phi(-\lambda z_1, \tilde{x}_2, \tilde{x}',  \tilde{y}_1, -\lambda z_2, \tilde{y}', z) \big\rangle\\
&= \big \langle T(z),  \int_{\tilde{x}_2, \tilde{x}', \tilde{y}_1, \tilde{y}'} \Phi(-\lambda z_1, \tilde{x}_2, \tilde{x}',  \tilde{y}_1, -\lambda z_2, \tilde{y}', 0,0,z') \big\rangle+ O(|\lambda|^{-1}),
\end{align*}
where $O(|\lambda|^{-1})$ is a current of mass $\lesssim |\lambda|^{-1}$ as $\lambda \to \infty$ because $(\tilde{A}_\lambda)_* T$ is of uniformly bounded mass on compact subsets of $U' \times\C^{\ell_1+\ell_2-k}$. 

Letting $\lambda \to \infty$ in the last equality,  we get a density current associated to $[V_1], [V_2],T$ in the left-hand side and a tangent current to $T$ along $V_1 \cap V_2$ in the right-hand side. The desired assertion thus follows. 
\endproof

In the above proof, we actually proved that there is a natural  1-1 correspondence between the set of tangent currents to $T$ along $V_1 \cap V_2$ and the set of density currents associated to  $T, [V_1],  [V_2]$. But the conclusion of Lemma \ref{le-sosanhRSvoiTRSdensity} is enough for our purpose later.


\section{Proof of Theorem  \ref{the-codim2V}} \label{sec-trunc1}

For every submanifold $Z$ of a manifold $Y$ with smooth boundary, denote by $[Z]$ the current of integration along $Z$. Let $\mu_r$ be the Haar measure on the boundary $\partial \D_r$ of $\D_r$.  Direct computations show that 
\begin{align}\label{eq-ddcDelta}
\ddc \int_1^r \frac{\d t}{t} [\D_t]= \mu_r - \mu_1,
\end{align}
recall $\dc:= \frac{i}{2\pi}(\bar \partial - \partial)$ and $\d \dc=\frac{i}{2\pi}\partial\bar\partial$. 

Let $X,Z,\mathcal{V},p_{1, \mathcal{V}},\omega,f$ be  as in the statement of Theorem \ref{the-codim2V}.  Put 
$$S_r:= c_r^{-1}\int_1^r \frac{\d t}{t} f_*[\D_t],$$
where $c_r:= T_f(r, \omega).$ 
We have the following equality
\begin{align} \label{eq-firstmain}
\dfrac{N_f(r, D)}{ T_f(r,\omega)}
=
1+
 \langle \varphi_D, \ddc S_r \rangle, 
\end{align}
where $[D]:= \ddc \varphi_D+ \omega$, which is known as the First Main Theorem in Nevanlinna's theory. Using (\ref{eq-ddcDelta}),  we get
\begin{align*}
\ddc S_r= c_r^{-1} f_* \mu_r - c_r^{-1} f_* \mu_1
\end{align*}
which implies that 
\begin{align}\label{ine-massddcSr}
\|\ddc S_r\|\lesssim c_r^{-1}.
\end{align}
 It follows that every limit current of the family $(S_r)_{r\in \R^+}$ as $r \to \infty$ is $\ddc$-closed. It is a well-known fact that at least such a limit current is $d$-closed.  Let $(r_k)_{k \in \N}$ be a sequence of positive real numbers converging to $\infty$ such that $S_{r_k}$ converges to a positive closed current $S$ as $k\to \infty$.
 
Since $p_{1, \mathcal{V}}$ is a submersion, the pull-back  $p_{1,\mathcal{V}}^* S$ is a well-defined current on $\mathcal{V}$. Here is an interpretation of the last current in terms of density currents.

\begin{lemma} \label{le-TVstep1codim2}
The Dinh-Sibony product $(p_1^* S) \curlywedge [\mathcal{V}]$ of $p_1^*S$ and  $[\mathcal{V}]$ is well-defined and equal to $p_{1,\mathcal{V}}^* S$. 
\end{lemma}

\proof We only need to work locally near $\mathcal{V}$. Let $U=U_1 \times U_2 \times U_3$ be a local chart of $X\times Z$ and $(x_1,x_2,x_3)$ be its coordinate system   such that $\mathcal{V}= \{x_3=0\}$, $p_1(x_1,x_2,x_3)=x_1$   and  $p_{1,\mathcal{V}}(x_1,x_2)=x_1$. Identify the normal bundle of $\mathcal{V}$ over $U$ with $U_1 \times U_2 \times \C^s$. We have $A_\lambda(x_1,x_2, x_3)= (x_1, x_2, \lambda x_3)$. Let $\Phi$ be a test function with compact support in $U$. Observe that 
\begin{align*}
\big\langle (A_\lambda)_* (p_1^* S), \Phi \big\rangle=  \big\langle S(x_1), \int_{x_2, x_3} \Phi(x_1, x_2, \lambda x_3) \big\rangle =\big\langle S(x_1), \int_{x_2, x_3} \Phi(x_1, x_2, \lambda x_3) \big\rangle
\end{align*}
which is equal to 
$$\big\langle S(x_1), \int_{x_2, x_3} \Phi(x_1, x_2,x_3) \big\rangle  = \big \langle \pi^* p_{1,\mathcal{V}}^* S, \Phi \big \rangle,$$
where $\pi$ is the natural projection from the normal bundle of $\mathcal{V}$ to $\mathcal{V}$. This finishes the proof.
\endproof

Let $X_a:= X \times \{a\}$ for $a\in Z$.

\begin{proposition}\label{prodensitycurrentSXa0codim2} There exists a countable union $\mathcal{A}$ of proper analytic subsets of $Z$ such that for $a \in Z \backslash \mathcal{A}$, we have that $\mathcal{V}_a$ is smooth and
\begin{align*}
S \curlywedge [\mathcal{V}_a]=0.
\end{align*}
\end{proposition}
\proof 
Put $\mathcal{V}'_a:= \mathcal{V} \cap X_a$. Recall $\mathcal{V}_a= p_1(\mathcal{V}'_a)$.   By Theorem \ref{the-densitybangkhong} and Lemma \ref{le-TVstep1codim2}, for $a$ outside a countable union $\mathcal{A}$ of proper analytic sets of $Z,$ the Dinh-Sibony product $(p_1^*S) \curlywedge [\mathcal{V}] \curlywedge [X_a]$ is zero because the bi-degree of $(p_1^*S) \curlywedge [\mathcal{V}]$ is $\dim X-1+s> \dim X= \dim X_a$.  Using the comment right after Theorem \ref{the-codim2V}, by enlarging $\mathcal{A}$ if necessary, one can assume that $\mathcal{V}$ is transverse to $X_a$ for $a\not \in \mathcal{A}$.  This together with  Lemma \ref{le-sosanhRSvoiTRSdensity} yields
\begin{align*}
(p_1^*S) \curlywedge [\mathcal{V}'_a]= p_1^* S \curlywedge [\mathcal{V} \cap X_a]=0
\end{align*}
for such $a$. Combining this and the fact that density currents associated to $S, [\mathcal{V}_a]$ are naturally identified with those associated to $p_1^* S$ and $[\mathcal{V}'_a]$, we obtain 
$$S \curlywedge [\mathcal{V}_a]=0$$
for $a \not  \in \mathcal{A}$.   This finishes the proof. 
\endproof

From now on, fix an $a \in Z\backslash \mathcal{A}$. For simplicity, we write $V$ for $\mathcal{V}_a$. Let $\rho: \tilde{X} \to X$ be the blowup of $X$ along $V$. Denote by $\tilde{V}$ the exceptional hypersurface of $V$. Let $\tilde{f}$ be the lift of $f$ to $\tilde{X}$, we then have $\rho \circ \tilde{f}= f$.

\begin{lemma}\label{le-uocluongTfngaomeganga}
\label{le-fngafwidehatr} There exists a  K\"ahler form $\tilde{\omega}$ on $\tilde{X}$ such that 
\[
 T_{\tilde{f}}(r,\tilde{\omega}) \le T_{f}(r, \omega)
 +
 O(1).
 \]
\end{lemma}

\proof  See \cite[2.5.1]{Huynh2016} for a proof. 
\endproof

Let $\tilde{S}_r$ be the pull-back of $S_r$ by $\rho$. Let $\tilde{S}$ be the strict transform of $S$ by $\rho$. 
 We have 
\[
\tilde{S}_r= c_r^{-1} \int_1^r \frac{\d t}{t} \tilde{f}_*[\D_t].
\]
Using this and Lemma \ref{le-fngafwidehatr} yields that  any limit current $\tilde{S}'$ of the sequence $(\tilde{S}_{r_k})$ is a $\ddc$-closed positive current on $\tilde{X}.$ This combined with the fact that  $\tilde{S}_{r_k} \to \tilde{S}$ outside $\tilde{\mathcal{V}}$ gives 
$$\tilde{S}'= \tilde{S}+ \tilde{S}'',$$
 where $\tilde{S}''$ is a $\ddc$-closed positive current of bi-dimension $(1,1)$ supported on $\tilde{V}.$ Hence, by a support theorem of Bassanelli \cite{Bassanelli}, $\tilde{S}''$ is a $\ddc$-closed current on $\tilde{V}$.  

Note that since $X$ is K\"ahler, so  is $\tilde{X}$.  By $\ddc$-Lemma, for every closed smooth $(n-1,n-1)$-form $\xi$ on $\tilde{V}$, where $n:= \dim X$,  the quantity $\langle \tilde{S}'',\xi \rangle$ depends only on the cohomology class of $\xi$. This combined with Serre's duality shows that  the cohomology class $\{\tilde{S}''\}$ of $\tilde{S}''$ in $H^{1,1}(\tilde{V})$ is well-defined.  

Let $\eta$ be a closed form in the cohomology class of $[\tilde{V}]$.  Recall that $\tilde{V}$ is naturally isomorphic to the fiber bundle $\P(F),$ where $F$ is  the normal bundle of $V$ in $X$. Denote by $\pi_{\P(F)}: \P(F) \to V$ the natural projection. The restriction of the cohomology class of $[\tilde{V}]$ to $\tilde{V}$ is the opposite of the Chern class $\omega_{\mathcal{O}_{\P(F)}(1)}$ of the line bundle $\mathcal{O}_{\P(F)}(1)$ which is the dual of the tautological line bundle of $\P(F)$.  

\begin{lemma}  \label{le-tildeSnomass}
$$\lim_{k \to \infty} \langle \tilde{S}_{r_k}, \eta \rangle  = 0.$$
\end{lemma}

We emphasize that the above lemma is not a direct consequence of the (semi-)continuity of total tangent classes given in \cite[Th. 4.11]{Dinh_Sibony_density}  because $\tilde{S}_r$ is not closed.  

\proof   By the First Main Theorem, we have
\begin{align} \label{eq-fngaDnga}
 \frac{N_{\tilde{f}}(r, \tilde{V})}{T_{f}(r, \omega)} \le  \langle \tilde{S}_r, \eta \rangle + O(1)/ T_f(r, \omega),
\end{align}
which yields
\begin{align} \label{ine-liminftidleS_r}
\liminf_{k \to \infty} \langle \tilde{S}_{r_k}, \eta \rangle \ge 0.
\end{align} 
By Proposition \ref{prodensitycurrentSXa0codim2} and (\ref{eq-cohoclassoftangentcurrent}),  we obtain that $\{\tilde{S}\} \smile \eta= 0$. To simplify the notation, we assume $\lim_{k\to \infty}\tilde{S}_{r_k}= \tilde{S}'$. Thus, 
\begin{align}\label{eqlimSreta}
\lim_{k \to \infty} \langle \tilde{S}_{r_k}, \eta \rangle= \langle \tilde{S}, \eta \rangle+  \langle \tilde{S}'', \eta \rangle= \langle \tilde{S}'', \eta \rangle= \int_{\tilde{V}} \{\tilde{S}''\} \wedge (\{\tilde{V}\}|_{\tilde{V}})= -\int_{\tilde{V}} \{\tilde{S}''\} \wedge  \omega_{\mathcal{O}_{\P(F)}(1)}.
\end{align}
On the other hand, since 
$$S+ \rho_* \tilde{S}''=  \rho_* \tilde{S}+\rho_* \tilde{S}''=  \lim_{k\to \infty}\rho_* \tilde{S}_{r_k} = \lim_{k\to \infty} S_{r_k}=S,$$
we get $\rho_* \tilde{S}''=0$. Since $\tilde{S}''$ is supported on $\tilde{V} \approx \P(F)$, we obtain 
\begin{align}\label{eqh-dimS''}
(\pi_{\P(F)})_*\tilde{S}''=0.
\end{align}
 By Leray's decomposition (see \cite{Bott_Tu}), we can write 
$$\{\tilde{S}''\}= \pi_{\P(F)}^*\kappa_0+ \pi_{\P(F)}^*\kappa_1 \smile  \omega_{\mathcal{O}_{\P(F)}(1)},$$
where $\kappa_j$ is a cohomology class of bidimension $(j,j)$ on $V$ for $j=0,1$. By (\ref{eqh-dimS''}), we obtain 
$$\kappa_1= (\pi_{\P(F)})_*\{\tilde{S}''\}=\{(\pi_{\P(F)})_*\tilde{S}''\}=0.$$
 This implies that 
 $\{\tilde{S}''\}= \pi_{\P(F)}^*\kappa_0.$ It follows particularly that $\kappa_0 \ge 0$. Combining this with (\ref{eqlimSreta}) gives
 $$\lim_{k \to \infty} \langle \tilde{S}_{r_k}, \eta \rangle= -\int_{\P(F)} \pi_{\P(F)}^*\kappa_0 \wedge  \omega_{\mathcal{O}_{\P(F)}(1)}= - \kappa_0\le 0.$$
This together with (\ref{ine-liminftidleS_r}) implies the desired equality.
\endproof

Recall $n= \dim X$.  Write 
$$\tilde{f}^* [\tilde{V}]= \sum_z \nu_{z, \tilde{f},\tilde{V}} \delta_z,$$
where $\delta_z$ is the Dirac mass at $z$.
Recall that $V$ is smooth.  We define $\nu_{z, f, V}$  as follows. Put $\nu_{z, f, V}:= 0$ if $z\not \in  f^{-1}(V)$. Consider now $z\in f^{-1}(V)$. Let $U$ be a local chart around $f(z)$ on $X$ and $x=(x_1,\ldots, x_n)$ a coordinate system on $U$ such that $V= \{x_j=0: 1\le j \le s\}$. Write $f=(f_1, \ldots, f_n)$ in these local coordinates. Let $\nu_{z, f, V}$ be the smallest number among the multiplicities of $z$ in the zero divisors of $f_j$ for $1 \le j \le s$ on $f^{-1}(U)$. This definition is independent of the choice of local coordinates. Recall that  $N_f(r, V)$ is the counting function of $f$ with respect to the divisor $\sum_z \nu_{z, f, V} \delta_z$ on $\C$.


\begin{lemma} \label{le-tinhhieuNN1boitangentcuren} We have 
\begin{align}\label{ine-multzDDnga}
\nu_{z,f,V}=  \nu_{z, \tilde{f}, \tilde{V}}
\end{align}
and
$$N_{\tilde{f}}(r, \tilde{V})= N_f(r, V).$$
\end{lemma}

\proof  
The second desired equality is a direct consequence of the first one. We prove now the first one.
Observe that $f^{-1}(V)= \tilde{f}^{-1}(\tilde{V})$ because $\rho \circ \tilde{f}= f$ and $\rho^{-1}(V)= \tilde{V}$. Hence it is enough to prove (\ref{ine-multzDDnga}) for $z\in f^{-1}(V)$.
Consider $z_0 \in f^{-1}(V)$.

 Let $(U,x)$ be  a local chart around $f(z_0)$ such that $V=\{x_j=0: 1 \le j \le s\}$. Write $f=(f_1, \ldots, f_n)$ as above. Denote by $\nu_{z,f_j}$ the multiplicity of $z$ in the zero divisor of $f_j$ on $f^{-1}(U)$ for $1 \le j \le s$. We have 
\begin{align}\label{eq-boiVboifjlocal}
\nu_{z_0, f,V}= \min_{1 \le j \le s}\{\nu_{z_0,f_j}\}.
\end{align}


Let us now recall how to construct the blowup $\tilde{X}$ along $V$ on $U$. Let $\tilde{U}:= \rho^{-1}(U)$. Let $w:=[w_1,\ldots,w_s]\in \P^{s-1}$. The set $\tilde{U}$ is the submanifold of $U \times \P^{s-1}$ given by the equations 
$$x_j w_\ell= x_\ell w_j$$
for $1\le j,\ell \le s$.   
Observe that $\tilde{f}(z)=\big(f(z), [f_1(z),\ldots, f_s(z)]\big) \in U \times \P^{s-1}$ for $z$ such that  $f(z) \in  U$.  

The set $\tilde{U}$ can be covered by $s$ standard local charts which we will describe  as follows. Let  $\tilde{U}_j$ be the subset of $\tilde{U}$ consisting of $(x,[w])$ with $w_j=1$ and $|w_\ell|\le c$ for $1 \le \ell \not =j \le s$, where $c$ is a constant big enough.  For $c$ big enough, the local charts $\tilde{U}_j$ cover $\tilde{U}$. Since the role of $\tilde{U}_j$ is the same, we now consider only $\tilde{U}_s$. The natural induced coordinates on $\tilde{U}_s$ are $(x, w_1, \ldots, w_{s-1})$ and $\tilde{f}=(f, f_1/f_s, \ldots, f_{s-1}/f_s)$ on $\tilde{f}^{-1}(\tilde{U}_s)$.  Hence, we have $|f_j/f_s| \le c$ on $\tilde{f}^{-1}(\tilde{U}_s)$ for $1 \le j \le s-1$.  The hypersurface $\tilde{V}$ is given by $x_s=0$ on $\tilde{U}_s$. 
We deduce that  if $\tilde{f}(z_0) \in \tilde{U}_s,$ then we must have   
$$\nu_{z_0, f,V}= \nu_{z_0, f_s}=\nu_{z_0, \tilde{f}, \tilde{V}}.$$
This finishes the proof.
\endproof

\begin{proof}[End of Proof of Theorem \ref{the-codim2V}]  
 
Let $a \in Z\backslash \mathcal{A}$ as above. By Lemmas \ref{le-tildeSnomass}, \ref{le-tinhhieuNN1boitangentcuren} and (\ref{eq-fngaDnga}), we get 
\begin{align*}
0\le \liminf_{r \to \infty} \frac{N_{f}(r, \mathcal{V}_a)}{T_{f}(r, \omega)}\le \liminf_{k \to \infty} \frac{N_{f}(r_k, \mathcal{V}_a)}{T_{f}(r_k, \omega)}= \liminf_{k \to \infty} \frac{N_{\tilde{f}}(r_k, \tilde{\mathcal{V}}_a)}{T_{f}(r_k, \omega)} =0.
\end{align*}
This finishes the proof.
\end{proof}

\section{Proof of Theorem \ref{thedefect0}} \label{sec-thjedefecto}

Let $X,L,f,E$ be as in the statement of Theorem \ref{thedefect0}. By using a basis of $E,$ we obtain an embedding from $X$ to a complex projective space.  So from now on, we can assume $X=\P^n$ and $L$ is the hyperplane line bundle of $X$. This reduction is not essential but it simplifies some computations. 

Consider first the case where $n\ge 2.$  Let $\Ta X$ be the tangent bundle of $X$. Let  $\widehat X:= \P(\Ta X)$ be the projectivisation of $\Ta X$ and $\pi: \widehat X \to X$ the natural projection. Let $\widehat f$ be the lift of $f$ to $\widehat X$ defined by  $\widehat f(z):= (f(z), [f'(z)]),$ where $f'$ is the derivative of $f$ and $z \in \C$. Hence $\widehat f$ is an entire curve in $\widehat X$. We also have $\widehat S_{r}, \widehat c_{r}$  for $\widehat f$ as $S_r, c_r$ for $f$.  

Let $\mathcal{O}_{\widehat X}(1)$ be the dual of the tautological line bundle of $\widehat X$. Let $\widehat \omega$ be a K\"ahler form on $\widehat X$ such that $\widehat \omega = \omega+ c\, \omega_{\mathcal{O}_{\widehat X}(1)}$, where $c$ is a strictly positive constant and  $\omega_{\mathcal{O}_{\widehat X}(1)}$ is a smooth Chern form of $\mathcal{O}_{\widehat X}(1)$ whose restriction to each fiber of $\pi$ is strictly positive.  Recall the following equality. 

\begin{lemma} \label{letautologi}We have
\begin{align} \label{eq-sosanhTfmuvoiTfngoaiLebhuuhan}
T_{\widehat f}(r, \widehat \omega)= T_f(r, \omega)+ o\big(T_f(r, \omega)\big)
\end{align}
as $r\to \infty$ outside a set of finite Lebesgue measure of $\R.$ 
\end{lemma}

\proof The desired assertion is equivalent to the equality
$$T_{\widehat f}\big(r,\mathcal{O}_{\widehat X}(1)\big)=  o\big(T_f(r, \omega)\big).$$
This is the tautological inequality of McQuillan \cite{McQuillan} which is in fact a consequence of the Lemma on logarithmic derivative. For the readers' convenience, we briefly recall how to prove it. 

Let $(\chi_j)$ be a partition of unity of $X$ subordinated to a finite covering $(U_j)$ of $X$ where $U_j$ are local charts on $X$. Trivialize  $\widehat X\approx U_j \times \P^{n-1}$ on $U_j.$ Let $x_j=(x_{j1}, \ldots, x_{jn})$ be the coordinates on $U_j$. Write $f(z)= (f_{j1}, \ldots, f_{jn})$ accordingly for $z\in f^{-1}(U_j).$ Put 
$$h_j(x_j,v):= \chi_j(x_j) \sum_{l=1}^n |v_l|^2,$$ 
where $[v] \in \P^{n-1}$. Thus $h:= \sum_j h_j$ is a Hermitian metric on $\mathcal{O}_{\widehat X}(1)$. This combined with the Lelong-Jensen formula gives
$$T_{\widehat f}\big(r, \mathcal{O}_{\P(\Ta X}(1)\big)= \int_{\partial \D_{r}} \log  \sum_j \chi_j(f) \sum_{\ell=1}^n |f'_{j\ell}|^2 d \mu_{r}+ O(1).$$
 Standard estimates in proofs of Lemma on logarithmic derivatives (see \cite[Le. 4.7.1]{Noguchi}) show that the last integral is $o\big(T_f(r, \omega)\big)$ as $r\to\infty$ outside a set of finite Lebesgue measure.  This finishes the proof. 
\endproof

We fix a sequence $\mathbf{r}:=(r_k)\in \R^+$ converging to $\infty$ such that  $\widehat S_{r_k}$ converges to a \emph{closed} positive current $\widehat S$  as $k \to \infty$ and (\ref{eq-sosanhTfmuvoiTfngoaiLebhuuhan}) holds for $r=r_k$.  
For $a=(a_0,\ldots,a_n)\in \C^{n+1}\backslash \{0\}$, put $\sigma_a(x):= \sum_{j=0}^n a_j x_j$, where $x=[x_0 :\ldots: x_n] \in X=\P^n$.  Identify $\P(E)$ with $\P^n$ via $a \longleftrightarrow \sigma_a.$  Denote by $D_a$ the hyperplane generated by $\sigma_a$.  Let $\pi_j: \widehat X\times \P(E)$ be the natural projections to its components for $j=1,2$.  

Let $\widehat D_a$ be the set of $(x,[v])\in \widehat X$ for $x \in D_a$ and $v \in \Ta_xX \backslash \{0\}$ tangent to $D_a$.  Let $U$ be a local chart of $X$ over which  $L$ is trivial. We identify $\sigma_j$ with functions on $U$ and $d \sigma_j$ with 1-form on $U$ (hence a function on $\Ta U$). Put 
$$H_1:= \big\{(x, [v],[a])\in \widehat X  \times \P^n:  \sigma_a(x) =0\big\}$$
and
$$\quad H_{2,U}:= \big\{(x, [v],[a])\in (\widehat X|_U) \times \P^n : \langle \d\sigma_a(x), v \rangle=0\big\}.$$
Let $U'$ be another local chart similar to $U$. Trivialize $L$ on $U'$.  Let $\sigma_a'$ be the trivialisation of $\sigma_a$ on $U'$. We have $\sigma_a'= g \sigma_a$ for some nowhere vanishing holomorhic function $g$  on $U \cap U'$ which is independent of $a$.  Thus $\d \sigma_a'= \d(g \sigma_a)= \d g \sigma_a + g \d \sigma_a.$ We deduce that  
 $H_1 \cap H_{2,U}=H_1 \cap H_{2,U'}$ on $\pi_1^{-1}\big(\Ta(U \cap U')\big)$.  Gluing $H_1\cap H_{2,U}$ together, we obtain  a  well-defined analytic subset $\mathcal{V}$ of  $\widehat X \times \P(E)$.

Let $p_{j, \mathcal{V}}$ be the restriction of $p_j$ to $\mathcal{V}$ for $j=1,2$. We have

\begin{lemma} \label{le-ktradkpVtruonghopdivisor} The set  $\mathcal{V}$ is a smooth submanifold of codimension $2$ of $X \times \P(E)$, the map $p_{1, \mathcal{V}}$ is a submersion and  the fiber of $p_{2,\mathcal{V}}$ at $a\in \P(E)$ is $\widehat D_a$.
\end{lemma}

\proof  The fact that the fiber of $p_{2,\mathcal{V}}$ at $a\in \P(E)$ is $\widehat D_a$ is clear from the construction.  It is sufficient to check the remaining desired assertions for $x=[x_0: \ldots : x_n]$ in  a local chart of $X=\P^n$. Hence, consider the local chart $U:= \{x\in \P^n: x_0=1\}$. We see that 
$$\mathcal{V}= \{a_0+ \sum_{j=1}^n a_j x_j=0, \sum_{j=1}^n a_j v_j =0 \}.$$
For $x, [v]$ fixed, these two defining equations of $\mathcal{V}$ give two hyperplanes in $\P(E)$ which are transverse to each other because $n\ge 2$. So $\mathcal{V}$ is smooth. Let $W_{j_0}$ be the local chart of $\P(E)$ containing $[a]$ with $a_{j_0}=1$. Consider first the case where $j_0\not =0$. Observing that $\mathcal{V}$ is the set of $(x,[v],[a])$ such that $x_{j_0}= -a_0 - \sum_{j\not= j_0} a_j x_j$ and $v_{j_0} = - \sum_{j \not = j_0} a_j v_j.$ So the map $p_{1, \mathcal{V}}$ can be identified with the map 
\begin{align*}
&\big(a_0, \ldots, a_{j_0-1}, a_{j_0+1}, \ldots, a_n, [v_1, \ldots, v_{j_0-1}, v_{j_0+1}, \ldots, v_n], x_1, \ldots, x_{j_0-1}, x_{j_0+1}, \ldots, x_n\big)\\
&\longrightarrow (x_1, \ldots, x_{j_0-1},     -a_0 - \sum_{j\not= j_0} a_j x_j, x_{j_0+1}, \ldots, x_n),
 \end{align*}
which is of maximal rank. The case where $j_0=1$ is treated similarly by observing that $(a_1, \ldots,a_n) \not = 0$ if $(x,[v],[a]) \in \mathcal{V}$ and $a_0=1$.  This finishes the proof. 
\endproof

Lemma \ref{le-ktradkpVtruonghopdivisor} combined with Proposition \ref{prodensitycurrentSXa0codim2} applied to $\widehat X$ in place of $X$ and $Z:= \P(E)$ gives

\begin{corollary} There exists a countable union $\mathcal{A}$ of proper analytic subsets of $\P(E)$ such that for $a\in \P(E) \backslash \mathcal{A}$, we have that 
$$\widehat S \curlywedge [\widehat D_a]= 0.$$
\end{corollary}

From now on fix  $a \in \P(E) \backslash \mathcal{A}$ and write $D, \widehat D$ for $D_a, \widehat D_a$ to simplify the notation.   Let  $\rho: \tilde{X} \to \widehat X$ be the blowup of $\widehat X$ along $\widehat D$. Let $\tilde{D}$ be the exceptional divisor of that blowup.  
Lift $\widehat f$ to a curve $\tilde{f}$ in $\tilde{X}$. Let $\tilde{\omega}$ be a K\"ahler form on $\tilde{X}$ such that 
\[
 T_{\tilde{f}}(r,\tilde{\omega}) \le T_{\widehat f}(r, \widehat \omega)
 +
 O(1),
 \]
see Lemma \ref{le-uocluongTfngaomeganga}. Let $\tilde{S}_r$ be the pull-back of $\widehat S_r$ by $\rho$. Let $\tilde{S}$ be the strict transform of $\widehat S$ by $\rho$. 
Let $\eta$ be a closed form in the cohomology class of $\tilde{D}$. By Lemma \ref{le-tildeSnomass} applied to $\widehat S, \widehat D$, we get
\begin{align} \label{eq-limitSngarkfhat}
\lim_{k \to \infty} \langle \tilde{S}_{r_k}, \eta \rangle  = 0.
\end{align}

\begin{lemma} \label{le-DDngapotential} We have
$$\langle \varphi_{\tilde{D}}, \ddc \tilde{S}_r \rangle= \langle \varphi_D, \ddc S_r \rangle+ c_r^{-1} O(1)$$
as $r \to \infty.$ 
\end{lemma}

\proof
Since $\varphi_{\tilde{D}}$ is a potential of $\tilde{D},$ we get 
$$|\varphi_{\tilde{D}}(\tilde{x})- \log \dist(\tilde{x}, \tilde{D})| \lesssim 1.$$
Thus 
$$\lim_{k\to \infty} \langle \log \dist(\tilde{x}, \tilde{D}), \ddc \tilde{S}_{r_k} \rangle =0.$$
Using the fact that $(\pi\circ \rho)(\tilde{D})=D$ gives 
$$|\log \dist(\tilde{x}, \tilde{D}) -  \log\dist(\pi\circ \rho (\tilde{x}), D) | \lesssim 1.$$
We also have  $(\pi\circ \rho)_* \tilde{S}_r= S_r$
because $\tilde{f}$ is a lift of $f$ to $\tilde{X}$. Recall that the mass of $\ddc \tilde{S}_{r}$ is $O(1) c_r^{-1}$ as $r \to \infty$.  It follows that 
$$\langle \log \dist(\tilde{x}, \tilde{D}), \ddc \tilde{S}_{r_k} \rangle= \langle \log \dist(x, D), (\pi\circ \rho)_*\ddc \tilde{S}_{r_k} \rangle+ O(1) c_r^{-1},$$
which is equal to  $\langle \log \dist(x, D),  \ddc S_{r_k} \rangle+ O(1)c_r^{-1}$  as $r \to \infty$.  The proof is finished.
\endproof

Write 
$$f^*[D]= \sum_z \nu_{z,f,D} \delta_z, \quad \tilde{f}^* [\tilde{D}]= \sum_z \nu_{z, \tilde{f}, \tilde{D}} \delta_z.$$
For $z \in \C$, let $\nu_{z,\widehat f, \widehat D}$ be the multiplicity of $z$ with respect to $\widehat{f}$, $\widehat{D}$ as in the setting of Lemma \ref{le-tinhhieuNN1boitangentcuren}.  Applying \eqref{ine-multzDDnga} to $\widehat f, \widehat D, \tilde{D},$ we obtain $\nu_{z,\widehat f, \widehat D}= \nu_{z, \tilde{f}, \tilde{D}}$.

In a local coordinates $(U,x)$ of $X,$ write $f(z)= \big( f_1(z), \ldots, f_n(z) \big)$. Let $f'_j$ be the derivative of $f_j$ for $1 \le j \le n$.  For $z_0\in \C,$ let  $\nu_{z_0,f'}$ be  the smallest non-negative integer such that $(z-z_0)^{-k}f'_j(z)$ for $1 \le j \le n$ are holomorphic functions which are not simultaneously zero at $z_0.$ This definition is independent of the choice of local charts. Thus the current $R_f:= \sum_{z_0\in \C} \nu_{z_0,f'}\delta_{z_0}$ is well-defined. Observe that the support of $f_* R_f$ consists of at most a countable number of points.  

\begin{lemma} \label{le-tinhhieuNN1boitangentcurennew} For $z\in \supp f^* D$, we have 
\begin{align}\label{ine-multzDDnganew}
\nu_{z,f,D} - 1=  \nu_{z, \widehat f,\widehat D}+ \nu_{z, f'}=  \nu_{z, \tilde{f},\tilde{D}}+ \nu_{z, f'}.
\end{align}
Consequently, for $D$ with $D\cap \supp f_* R_f =\varnothing$, there holds 
$$N_{\tilde{f}}(r, \tilde{D})= N_{\widehat f}(r, \widehat D) = N_f(r, D)-N^{[1]}_{f}(r, D).$$
\end{lemma}

Note that since $\supp f_* R_f$ is an at most countable set,  the set of $D$ for which $D \cap \supp f_* R_f  \not =\varnothing$ is a countable union of hyperplanes in $\P(E)$.

\proof We already proved the second inequality of (\ref{ine-multzDDnganew}). It remains to prove the first one.  Consider now $z_0 \in \supp f^* [D]$. Thus $\nu_{z_0,f,D}\ge 1$.

 Let $(U,x)$ be  a local chart around $f(z_0)$ such that $x=(x_1,\ldots, x_n)$ and $D=\{x_1=0\}$.  We then obtain an induced coordinate system $(x,v)$ on $TX|_U= U \times \C^n$ and $\widehat D$ is given by $x_1= v_1 =0$ there.  Write $f=(f_1, \ldots, f_n)$ in these coordinates. Hence $f^*[D]$ is the divisor generated by $f_1$ on $f^{-1}(U)$.  In these coordinates, we have $\P(\Ta U)=U \times \P^{n-1}$. 

Put $f':=(f'_1, \ldots, f'_n)$.  We have $\widehat f=(f, [f']) \in U \times \P^{n-1}$ which is a holomorphic curve. Without loss of generality, since $(z-z_0)^{-\nu_{z_0,f'}}f'_j(z)$ isn't zero at the same time at $z_0$ for $1 \le j \le n$, we can assume that $(z-z_0)^{-\nu_{z_0,f'}}f'_n(z)$ is not zero at $z=z_0$. 

 Using standard local charts on $\P^{n-1}$, we can cover $U\times \P^{n-1}$ by a finite number of local charts. Let $\widehat U$ be the local chart with $v_n=1$. So the coordinates on $\widehat U$ are $(x, v_1, \ldots, v_{n-1})$ and $\widehat D$ is still given by $x_1= v_1 =0$. In these coordinates, $\widehat f= \big(f, g_1, \ldots, g_{n-1}\big)$,  where $g_j:= f'_j/ f'_n$ for $1 \le j \le n-1$.  Observe that  the order of $g_1$ at $z_0$ is $(\nu_{z_0,f,D}-1- \nu_{z_0,f'})< \nu_{z_0,f,D}$. Thus, 
$$\nu_{z_0,\widehat f,\widehat D}=\nu_{z_0,f,D}-1- \nu_{z_0,f'}.$$
So (\ref{ine-multzDDnganew}) follows.   This finishes the proof.
\endproof

\begin{proposition} \label{pro-hieuTfvaN1bieudhien} For $D$ with $D\cap \supp f_* R_f =\varnothing$, we have 
\[
1- \frac{N^{[1]}_{f}(r, D)}{T_f(r, \omega)}=  \langle \tilde{S}_r, \eta \rangle + \dfrac{O(1)}{ T_f(r,\omega)}
\eqno
{{\scriptstyle (r\,\to\, \infty)}.}
\]
\end{proposition}

\proof  Starting from First Main Theorem and using Lemma \ref{letautologi}, we get 
\begin{align*} 
 \frac{N_{\tilde{f}}(r, \tilde{D})}{T_{f}(r, \omega)}
 &= \langle \tilde{S}_r, \eta \rangle + \langle \varphi_{\tilde{D}}, \ddc \tilde{S}_r \rangle\\
  \explain{Use Lemma \ref{le-DDngapotential}}
 \ \ \ \ \ \ \ \ \ \ \ \ \
 &
 =
\langle \tilde{S}_r, \eta \rangle
 +
 \langle \varphi_{D}, \ddc S_r \rangle+ O(1)[T_{f}(r,  \omega)]^{-1}\\
 \explain{Use First Main Theorem}
\ \ \ \ \ \ \ \ \ \ \ \ \
&=
\langle \tilde{S}_r, \eta \rangle
+
\frac{N_{f}(r, D)}{T_{ f}(r,  \omega)}
-1
 + \dfrac{O(1)}{ T_f(r,\omega)}\\
 \explain{Use Lemma \ref{le-tinhhieuNN1boitangentcurennew}}
 \ \ \ \ \ \ \ \ \ \ \ \ \
 &=
 \langle \tilde{S}_r, \eta \rangle
 + \frac{N_{\tilde{f}}(r,\tilde{D})+N_f^{[1]}(r, D)}{T_{f}(r, \omega)}
 -
 1
 + \dfrac{O(1)}{ T_f(r,\omega)},
\end{align*}
which yields the desired equality.
\endproof

\begin{proof}[End of the proof of Theorem \ref{thedefect0} when $n\ge 2$]  Recall that we reduced the problem to the case where $X= \P^n$ and $L$ is the hyperplane line bundle. We also have fixed a hyperplane $D$ on $X$ such that  the tangent current of $\widehat S$ along $\widehat D$ is zero and $\supp D \cap \supp f_* R_f= \varnothing$.  We had also shown that  the set of such $D$ contains the complement of a countable union  of proper analytic subsets of $\P(E)$. By Proposition \ref{pro-hieuTfvaN1bieudhien} and (\ref{eq-limitSngarkfhat}), one gets
$$T_f(r_k, \omega)- N^{[1]}_f(r_k,D)= o\big(T_f(r_k,\omega)\big)$$
as $k\to \infty.$ Hence $\delta^{[1]}_f(D)=0$.  This finishes the proof.   
\end{proof}

Finally, to finish the proof of Theorem \ref{thedefect0}, we treat the remaining case where $n=1$, \emph{i.e.,} $X=\P^1$. In this case, Theorem \ref{thedefect0} is a direct consequence of the following general result.

\begin{lemma} \label{le-truncationn} Let $f: \C \to \P^n$ be a  non-constant holomorphic curve. There exists a countable union $\mathcal{A}$ of proper linear subspaces of the space of hyperplanes of $\P^n$ such that  for any hyperplane $D\not \in \mathcal{A}$, we have  
$$\delta^{[n]}_f(D)=0.$$
	\end{lemma}
\begin{proof} Without loss of generality, we can assume that $f$ is linearly non-degenerate because otherwise we can consider the smallest linear subspace of $\P^n$ containing $f(\C)$.  Let $(D_j)_{j=1}^q$  be a family of hyperplanes in general position. By Cartan's Second Main Theorem \cite{Cartan1933}, we get the following defect relation 
\begin{align} \label{ine-defectrelationhyper}
\sum_j^q \delta^{[n]}_f(D_j) \le n+1.
\end{align}
For any positive number $k$, set $\mathcal{A}_k:=\{D:  \delta^{[n]}_f(D)\geq 1/k\}$. It is clear that the desired equality holds true for all $D\notin\cup_{k=1}^{\infty}\mathcal{A}_k$. Thus the problem reduces to proving that for each integer number $k$, the set $\mathcal{A}_k$ is contained in a countable union of proper linear subspaces of $\P^n$.
	
Suppose on the contrary that this is not the case  for some $k$. For any positive integer $q$, we now construct a family of hyperplanes $\{D_i\}_{1\leq i\leq q}\subset \mathcal{A}_k$ in general position in $\mathbb{ P}^n$. 
	
We first taking a divisor $D_0\in\mathcal{A}_k\setminus\{0\}$ and consider the linear subspace $\mathcal{E}_0$ of $E$ generated by $D_0$. Since $\mathcal{E}_0$ is a proper linear subspace of $\P^n$, $\mathcal{A}_k\setminus \mathcal{E}_0$ is nonempty. Taking $D_1\in\mathcal{A}_k\setminus \mathcal{E}_0$ and let $\mathcal{E}_1$ be the vector space generated by $D_0,D_1$. Again since $\mathcal{E}_1$ is a linear space, $\mathcal{A}_k\setminus \mathcal{E}_1$ is not empty and we can pick in this set a hyperplane $D_2$. Iterating this process, we obtain at the $(q+1)^{\text{th}}$--step a family $\{D_i\}_{0\leq i\leq q}\subset \mathcal{A}_k$ of $q+1$ divisors in general position in $\mathbb{ P}^n$. Applying (\ref{ine-defectrelationhyper}) to this family gives
	$$n+1 \ge q/k.$$
Letting $q\to \infty$ yields a contradiction. This finishes the proof.
\end{proof}

\bibliography{biblio_family_MA,biblio_Viet_papers}
\bibliographystyle{siam}

\bigskip

\noindent
\Addresses
\end{document}